\documentclass[12pt,a4paper]{amsart}
\usepackage{amsmath}
\usepackage{amssymb}
\usepackage{amsthm}
\usepackage{amsfonts}
\usepackage{mathrsfs}
\usepackage{color}
\usepackage[colorlinks,linkcolor=blue,anchorcolor=green,citecolor=red]{hyperref}
\usepackage{cleveref}
\usepackage[margin=1.0in]{geometry}
\usepackage[fontsize=13pt]{fontsize}
\usepackage{esint}

\numberwithin{equation}{section}

\newtheorem{theorem}{Theorem}[section]
\newtheorem{lemma}[theorem]{Lemma}

\newtheorem{proposition}[theorem]{Proposition}

\theoremstyle{definition}
\newtheorem{definition}[theorem]{Definition}

\theoremstyle{remark}
\newtheorem{remark}[theorem]{Remark}

\theoremstyle{proof}

\newcommand{\db}{\bar{\partial}}
\newcommand{\dd}{\partial}

\newcommand{\rr}[1]{\mathrm{#1}}

\newcommand{\vol}{\rr{Vol}}

\newcommand{\leqs}{\leqslant}
\newcommand{\geqs}{\geqslant}

\newcommand{\rean}{\mathbb{R}}
\newcommand{\cpxn}{\mathbb{C}}
\newcommand{\hypern}{\mathbb{H}}
\newcommand{\proj}{\mathbb{P}}
\newcommand{\iv}{^{-1}}
\newcommand{\oo}{\mathcal{O}}

\newcommand{\kah}{K\"{a}hler }
\newcommand{\ii}{\sqrt{-1}}
\newcommand{\hp}{H_{par}}
\newcommand{\ddp}{\partial_{par}}

\newcommand{\lb}{\mathcal{L}}
\newcommand{\eb}{\mathcal{E}}
\newcommand{\fb}{\mathcal{F}}

\newcommand{\vb}{\mathcal{V}}

\newcommand{\mbar}{\overline{M}}

\DeclareMathOperator{\tr}{tr}

\DeclareMathOperator{\codim}{codim}

\DeclareMathOperator{\id}{id}

\DeclareMathOperator{\End}{End}
\DeclareMathOperator{\rank}{rank}

\DeclareMathOperator{\pardeg}{pardeg}
\DeclareMathOperator{\ord}{ord}
\DeclareMathOperator{\Bl}{Bl}
\DeclareMathOperator{\supp}{supp}
\DeclareMathOperator{\spann}{Span}
\DeclareMathOperator{\Aut}{Aut}
\DeclareMathOperator{\diag}{diag}

\title{Hermitian-Einstein Metrics on Parabolic Bundles over compact complex surfaces}

\author{Xilun Li}
\address{School of Mathematical Sciences\\
Peking University\\
Beijing\\ 100871\\ China}
\email{lxl28@stu.pku.edu.cn}
\author{Gang Tian}
\address{Beijing International Center for Mathematical Research, Peking University, Beijing 100871, China}
\email{gtian@math.pku.edu.cn}

\thanks{}
\date{}
\dedicatory{}

\begin{document}

\maketitle

\begin{abstract}
    We prove the Kobayashi-Hitchin correspondence for parabolic bundles over compact non\kah surfaces with simple normal crossing divisor or compact non\kah manifolds of any dimension with smooth divisor.
\end{abstract}

\maketitle
\tableofcontents

\section{Introduction}

The celebrated Kobayashi-Hitchin correspondence demonstrates the equivalence of the existence of Hermitian-Einstein metric on a holomorphic vector bundle and slope stability. Moreover, this identification is not only as a bijection of two sets, but also as an analytic isomorphism of moduli spaces\cite{MR1370660}. This correspondence is beautiful and important, and has many interesting extensions, such as vortices case, Higgs bundles and so on, see \cite[p.13-17]{MR1370660} for a detailed introduction.

One of useful extensions is to consider Hermitian-Einstein metric with mild singularities along some divisors $D\subseteq \mbar$. This leads to the concept of parabolic bundles. Roughly speaking, the metric on parabolic bundles is a Hermitian metric on $\eb|_M$ of polynomial growth and with prescribed residues around the divisor $D$, where $M:=\mbar\backslash D$. The parabolic structure was introduced by Mehta and Seshadri\cite{MR575939}, as a generalization of Narasimhan-Seshadri's theorem\cite{MR184252} to noncompact Riemann surface with finite volume case. They show that the irreducible unitary representation of discrete subgroup $\Gamma\subseteq \Aut\hypern$ with finite volume $\hypern/\Gamma$ can be identified to parabolic stable bundle over $\overline{\hypern/\Gamma}$ of parabolic degree $0$. 

It's natural to generalize the Kobayashi-Hitchin correspondence of parabolic bundle version to higher dimension and non\kah cases. \cite{MR1373062,MR1701135} proved the case of compact \kah surfaces and smooth divisor, depending on Simpson's work\cite{MR944577} on Kobayashi-Hitchin correspondence of Higgs bundle over noncompact manifolds. \cite{MR1863850} proved the case of compact \kah surfaces and divisor with normal crossings, using an orbifold argument. \cite{MR1775134} proved the case of compact \kah manifolds of any dimension and divisor with normal crossing, but the reference \kah metric $\omega$ needs to be changed to conical metric of cone angle $2\pi\alpha$, i.e. for $\alpha\in (0,1)$,
\begin{align*}
    \omega_{\alpha}:=\omega+\ii\varepsilon_\alpha\sum_{i=1}^m\dd\db||\sigma_i||^{2\alpha},
\end{align*}
where $\sigma_i$ is the canonical section of
$\oo_X(D_i)$ which vanishes on $D_i$. \cite{MR2310103} proved the parabolic Higgs version for smooth projective varieties with simple normal crossings.  For non\kah case, the standard Kobayashi-Hitchin correspondence was proved by \cite{MR939923} for surface case and \cite{MR915839} for any dimension. The Simpson's work\cite{MR944577} on Higgs bundle was generalized to non\kah case by \cite{MR4237961}.


The arguments in \cite{MR2310103} run the Hermitian-Yang-Mills flow with respect to the conical metric $\omega_{\varepsilon}$ and take a limit to get the Hermitian-Einstein metric with respect to $\omega$. We note that introducing conical metrics is to improve the regularity of $\Lambda_\omega F_H$, where $H$ is the initial metric on $\eb|_M$. Thus to avoid using the conical metrics, we need to construct the initial metric $\hp$ compatible with the parabolic structure carefully such that we have the enough regularity.


In this paper, we construct a parabolic Hermitian metric $\hp$ with bounded $\Lambda_\omega F_{\hp}$ for any parabolic bundles over non\kah surfaces with simple normal crossing divisor, or non\kah manifolds of any dimension with smooth divisor. The construction doesn't need the stability of the parabolic structure. As a consequence, we prove the following parabolic Kobayashi-Hitchin correspondence for non\kah surfaces with simple normal crossing divisor. Our argument can be easily generalized to compact non\kah manifolds of any dimension with smooth divisor.
\begin{theorem}[cf. \Cref{thm-KHset2}]\label{thm-KHset}
    Let $(\mbar,\omega)$ be a compact complex surface with a Gauduchon metric and $D=\sum_{i=1}^m D_i$ be a simple normal crossing divisor. 
    $\eb$ is a holomorphic vector bundle over $\mbar$ with a parabolic structure. Denote $M:=\mbar\backslash D$.

    If $\eb$ is parabolic stable, then there exists a Hermitian-Einstein metric on $\eb|_M$ compatible with
    the parabolic structure with respect to $\omega$.

    If $\eb|_M$ admits a Hermitian-Einstein metric compatible with the parabolic structure with respect to $\omega$ and $\eb$ is indecomposable,
    then $\eb$ is parabolic stable. 
\end{theorem} 
Our proof is also based on results of noncompact non\kah Higgs bundle\cite{MR4237961} and the equivalence of parabolic stability and Simpson's analytic stability\footnote{During the preparation of this paper, we noticed a recently posted paper \cite{jiangli2025} in which the authors improved \cite{MR1775134} for smooth reference \kah metric $\omega$ rather than conical metric $\omega_\alpha$. Their arguments are similar to those in \cite{MR2310103} and may give an alternative proof of Theorem 1.1.}.
\begin{theorem}
     Let $(\mbar^n,\omega)$ be a compact complex manifold with a Gauduchon metric and $D$ be a smooth divisor.
    $\eb$ is a holomorphic vector bundle over $\mbar$ with a parabolic structure. Denote $M:=\mbar\backslash D$.

    If $\eb$ is parabolic stable, then there exists a Hermitian-Einstein metric on $\eb|_M$ compatible with
    the parabolic structure with respect to $\omega$.

    If $\eb|_M$ admits a Hermitian-Einstein metric compatible with the parabolic structure with respect to $\omega$ and $\eb$ is indecomposable,
    then $\eb$ is parabolic stable. 
\end{theorem}
The proof of Theorem 1.2 is identical to that of Theorem 1.1\footnote{It is unclear if the arguments in \cite{jiangli2025} can work for non\kah manifolds of high dimensions. The subtlity is that one needs Gauduchon metrics in order to define the degree of coherent sheaves. However, given an non\kah Gauduchon metric $\omega$, the conical one $\omega_\alpha$ may not be Gauduchon when $\dim_\cpxn M\geqs3$. On the other hand, if we modify $\omega_\alpha$ to be Gauduchon, the degree may be changed. In a forthcoming paper, we will extend Theorem 1.2 to non\kah manifolds and divisors with normal crossing.}.

The organization of this paper is as follows: 

In Section 2, we review some concepts of the parabolic bundle, and construct an initial metric $\hp$ with $|\Lambda_\omega F_{\hp}|_{\hp}\in L^\infty$. Moreover, we state the main theorem.

In Section 3, we introduce Simpson's assumptions and concepts of analytic stability. Then we show the equivalence of parabolic stability and analytic stability.

In Section 4, we give the proof of \Cref{thm-KHset}.

In Section 5, we prove some lemmas used in the paper as an Appendix.
\newline

\textbf{Acknowledgements} The first author would like to thank Zexuan Ouyang, Xin Fu, Minghao Miao and Shengxuan Zhou for helpful discussions. Authors are supported by National Key R\&D Program of China 2020YFA0712800.

\section{Preliminary}
    Let $(\mbar,\omega)$ be a compact complex surface admitting a Gauduchon 
metric, i.e. $\dd\db\omega^{n-1}=0$, where $n=\dim_\cpxn\mbar=2$.  Let $D=\sum_{i=1}^m D_i$ be a simple normal crossing divisor
in $\mbar$, and we take $M:=\mbar\backslash D$. $\eb$ is a holomorphic vector bundle over $\mbar$
of rank $r$. 
\begin{remark}
    By Gauduchon's result\cite{MR0742896}, each 
conformal class of a Hermitian metric of $\mbar$ admits a Gauduchon metric,
which is unique up to a multiple constant.
\end{remark}
\begin{definition}
    A \emph{parabolic bundle} is a holomorphic vector bundle $\eb$ with a parabolic
structure with respect to $D$, which consists of
    \begin{itemize}
        \item flags of $\eb|_{D_i}$:
            \begin{align*}
                \eb|_{D_i}=\fb_1^i\supseteq \fb_2^i\supseteq\cdots\supseteq \fb^i_{m_i}\supseteq 0,
            \end{align*}
            where $\fb_{j+1}^i\subseteq \fb_j^i$ is a proper subbundle. Moreover, the
            flags satisfy the compatibility condition: For each $1\leqs i_1,i_2\leqs m$,
                $\left\{\fb^{i_k}_j|_{D_{i_1}\cap D_{i_2}}:k\in \{1,2\}, 1\leqs j\leqs m_{i_k}\right\}$
            is a refined flag of $\left\{\fb^{i_k}_j|_{D_{i_1}\cap D_{i_2}}: 1\leqs j\leqs m_{i_k}\right\}$
            for each $k\in\{1,2\}$.
        \item weights $0\leqs\alpha_1^i<\alpha_2^i<\cdots<\alpha_{m_i}^i<1$ attached
            to $\fb^i_1,\cdots,\fb^i_{m_i}$.
    \end{itemize}
\end{definition}
\subsection{Parabolic stability}
To define the parabolic degree, we first recall the definition of the classical
degree of a coherent sheaf.
\begin{definition}
    Let $\lb$ be a holomorphic line bundle over $\mbar$, then the \emph{degree}
    of $\lb$ with respect to $\omega$ is defined by 
    \begin{align*}
        \deg_{\omega}\lb:=\frac{\ii}{2\pi}\int_{\mbar}\tr F_{h}\wedge\omega^{n-1},
    \end{align*}
    where $h$ is a Hermitian metric of $\lb$, $F_h$ is the curvature of Chern 
    connection. For a general coherent sheaf $\fb$ of rank $r$, $\det\fb:=(\wedge^r\fb)^{**}$
    is always a holomorphic line bundle, then the degree is defined by 
    \begin{align*}
        \deg_{\omega}\fb:=\deg_{\omega}(\det\fb).
    \end{align*}
\end{definition}
\begin{remark}
    For two Hermitian metric of $\lb$, $\tr F_h$ differs a $\dd\db$-exact term. It induces a cohomological class $c_1^{BC}(\lb):=[\frac{\ii}{2\pi}\tr F_{h}]\in H^{1,1}_{BC}(\mbar,\rean)$, where the Bott-Chern cohomology is defined by
    \begin{align*}
        H^{1,1}_{BC}(\mbar,\rean):=\frac{\{\alpha\in\Lambda^{1,1}(\mbar,\rean):d\alpha=0\}}{\{\ii\dd\db f:f\in C^\infty(\mbar,\rean)\}}.
    \end{align*}
    Since $\omega$ is a Gauduchon metric, the degree is well defined, which is independent on the choice of the Hermitian metric.
\end{remark}
\begin{definition}
        The \emph{parabolic degree} of a parabolic bundle $\eb$ with respect to $\omega$
    is defined by 
    \begin{align*}
        \pardeg_{\omega} \eb:=\deg_{\omega}\eb+\sum_{i=1}^m\sum_{j=1}^{m_i}\rank(\fb^i_j/ \fb^i_{j+1})\alpha^i_j\deg_{\omega} [D_i],
    \end{align*}
    where $[D_i]$ is the line bundle induced by $D_i$.
\end{definition}
Suppose $\vb$ is a coherent subsheaf of $\eb$ with torsion free quotient.
Then there exists a natural induced parabolic structure:
\begin{itemize}
    \item The flag of coherent subsheaves
    \begin{align*}
        \vb|_{D_i}=\fb^i_1\vb\supseteq \cdots\supseteq \fb^i_{n_i}\vb\supseteq 0
    \end{align*}
    is induced by $\fb^i_1\cap \vb\supseteq\cdots\supseteq \fb^i_{m_i}\cap \vb\supseteq 0$.
    \item The weights attached to the flag are defined by
    \begin{align*}
        \beta^i_j:=\max\{\alpha^i_k:\fb^i_j\vb\subseteq \fb^i_k\cap\vb\}.
    \end{align*}
\end{itemize}
Thus we can also define the parabolic degree of the subsheaf $\vb$:
\begin{align*}
    \pardeg_{\omega} \vb:=\deg_{\omega}\vb+\sum_{i=1}^m\sum_{j=1}^{n_i}\rank(\fb^i_j\vb/ \fb^i_{j+1}\vb)\beta^i_j\deg_{\omega} [D_i],
\end{align*}
\begin{remark}
    Here we abuse the terminology a little since we only construct a flag of coherent subsheaves rather than subbundles. Nevertheless, it's sufficient to define the parabolic degree and parabolic stability. 
\end{remark}
Now we can define the parabolic stability.
\begin{definition}
    A parabolic bundle $\eb$ is called \emph{parabolic (semi-)stable} with respect to $\omega$ if for any proper coherent subsheaf $\vb$ of $\eb$ with torsion free quotient, we have
\begin{align*}
    \frac{\pardeg_{\omega}\vb}{\rank \vb}(\leqs)<\frac{\pardeg_{\omega}\eb}{\rank\eb}.
\end{align*}
\end{definition}

\subsection{Metrics compatible with the parabolic structure}

Since the parabolic Hermitian-Einstein problem looks for the canonical metric on the holomorphic bundle $\eb|_{M}$ over the noncompact manifold $M$, it's necessary to determine the asymtotic behaviour of the metric near $D$. To do this, we construct an initial Hermitian metric $H_{par}$ with polynomial growth compatible with the parabolic structure.

First we fix a smooth Hermitian metric $H_1$ on $\eb$ over $\mbar$, which is determined later. Let $U_i$ be the tubular neighborhood of $D_i$, and $U:=\bigcup U_i$. We assume $F^i_j$ are smooth Hermitian vector bundles over $U_i$ such that $F^i_j|_{D_i}=\fb^i_j$ and $F^i_{j+1}\subset F^i_j$, which are also determined later. By taking orthogonal complements with respect to $H_1$, we have smooth Hermitian vector bundles $Q^i_j (1\leqs j\leqs m_i)$ over $U_i$ such that they are mutually vertical and $\bigoplus_{l\geqs j}Q^i_l=F^i_j$ for any $j$. Thus we have a smooth orthogonal decomposition
	\begin{align*}
		(\eb|_{U_i},H_1)=(Q^i_1,H_1|_{Q^i_1})\oplus\cdots\oplus (Q^i_{m_i},H_1|_{Q^i_{m_i}}).
	\end{align*}
 Now we can define a Hermitian metric $\hp$ on $\eb|_M$ by taking
	\begin{align*}
		H_{par}|_{U_i\backslash D}:=\bigoplus_j ||\sigma_1||^{2\alpha^1_j}\cdots||\sigma_m||^{2\alpha^m_j}H_1|_{Q^i_j},
	\end{align*}
and extending $H_{par}$ on $M\backslash U$ smoothly. To make it well-defined, we need to check $H_{par}|_{U_i\backslash D}$ and $H_{par}|_{U_j\backslash D}$ coincide on $U_i\cap U_j$. We first compute its curvature:
\begin{proposition}\label{prop-Fcurvature-compute}
	Let $\hp$ be the metric defined above, $p\in V\subset U_i\backslash(\cup_{j\neq i}U_j)$ and $p\notin D$. $\{e_1,\cdots,e_r\}$ are the local orthonormal frame of $\eb|_V$ with respect to $\eb|_V=\bigoplus Q^i_j|_V$. Let $\nabla=\ddp+\db_\eb$ be the Chern connection of $\hp$ and $\db_\eb e_j=b_j^k e_k$, where $b_j^k\in A^{0,1}(\overline{V})$. We assume $b_j^k|_D=0$ for any $\alpha^i_j>\alpha^i_k$.
	Then
	\begin{enumerate}
		\item $|F_{\hp}|_{\hp}=O(||\sigma_i||^{-\gamma})$, where $\gamma=\max\{\alpha^i_{j+1}-\alpha^i_j,1-(\alpha^i_{j+1}-\alpha^i_j)\}$. So $F_{\hp}\in L^p(M)$ for some $p>2$.
		\item $|\Lambda_\omega F_{\hp}|_{\hp}=O(1)$ if the following holds:
			\begin{align*}
				\Lambda_\omega\dd b_j^k=O(||\sigma_i||), \Lambda_\omega(d\sigma_i\wedge b_j^k)=O(||\sigma_i||^2), \Lambda_\omega(d\sigma_i\wedge b_k^j)=O(||\sigma_i||) \text{ for any }\alpha^i_j>\alpha^i_k.
			\end{align*}
	\end{enumerate}
\end{proposition}
\begin{proof}
	For simplicity of the notation, we assume $m_i=r$, i.e. each of the weights has multipicity 1. General case is quite similar. Let $V=\{(z_1,\cdots,z_n):|z_k|<1,\forall k\}$ and assume $||\sigma_i||^2=f|z_1|^2$ for some smooth function $f$ with $\log f\in C^\infty(\overline{V})$. Since $V\subset U_i\backslash(\cup_{j\neq i}U_j)$, we can write $\hp|_V$ as
	\begin{align*}
		\hp|_V=f_1||\sigma_i||^{2\alpha_1}h_1\oplus\cdots\oplus f_r||\sigma_i||^{2\alpha_r}h_r,
	\end{align*}
	where $\log f_j\in C^\infty(\overline{V})$, $\alpha_j=\alpha^i_j$ and $h_j=H_1|_{Q^i_j}$. Suppose $\ddp e_j=a_j^k e_k$, where $a_j^k\in A^{1,0}(V\backslash D)$. Since Chern connection is compatible with the Hermtian metric, we have
	\begin{align*}
		\dd\hp(e_j,e_k)=\hp(a_j^le_l,e_k)+\hp(e_j,b_k^le_l).
	\end{align*}
	By $\hp(e_j,e_k)=f_j||\sigma_i||^{2\alpha^i_j}\delta_{jk}$, we have
	\begin{align*}
		a_j^k=\delta_{jk}\dd\log f_j+\alpha_j\delta_{jk}(\dd\log f+z_1\iv dz_1)-\overline{b}_k^j||\sigma_i||^{2(\alpha_j-\alpha_k)}f_jf_k\iv.
	\end{align*}
	By definition, we have
	\begin{align*}
		F e_j=&\nabla^2e_j=(\ddp\db_\eb+\db_\eb\ddp)e_j
		=(\dd b_j^k+\db a_j^k-b_j^l\wedge a_l^k-a_j^l\wedge b_l^k)e_k.
	\end{align*}
	To compute the norm with respect to $\hp$, we use another frame $\{\tilde{e}_1,\cdots,\tilde{e}_r \}$ of $\eb|_{V\backslash D}$, where $\tilde{e}_j:=||\sigma_i||^{-\alpha_j}e_j$. Then the curvature tensor is $F_{\hp}\tilde{e}_j=\tilde{F}_j^k\tilde{e}_k$, where
	\begin{align*}
		\tilde{F}_j^k=&\dd b_j^k||\sigma_i||^{\alpha_k-\alpha_j}-\db\overline{b}^j_k||\sigma_i||^{\alpha_j-\alpha_k}f_jf_k\iv+(\alpha_j-\alpha_k)\overline{b}^j_k\wedge(\db\log f+\bar{z}_1\iv d\bar{z}_1)||\sigma_i||^{\alpha_j-\alpha_k}f_jf_k\iv\\
		&-(\alpha_j-\alpha_k)(\dd\log f+z_1\iv dz_1)\wedge b_j^k||\sigma_i||^{\alpha_k-\alpha_j}+\overline{b}^j_k\wedge\db(f_j f_k\iv)||\sigma_i||^{\alpha_j-\alpha_k}\\
		&-\dd\log(f_j f_k\iv)\wedge b_j^k||\sigma_i||^{\alpha_k-\alpha_j}+f_jf_l\iv||\sigma_i||^{\alpha_j+\alpha_k-2\alpha_l}\overline{b}^j_l\wedge b^k_l+f_l f_k\iv||\sigma_i||^{2\alpha_l-\alpha_j-\alpha_k}b^l_j\wedge\overline{b}^l_k.
	\end{align*}
	We have $\overline{\tilde{F}^j_k}=-f_kf_j\iv\tilde{F}^k_j$, so we only need to consider $j\geqs k$. Since we assume $b^r_s|_D=0$ for any $s>r$, we have $b^r_s=O(||\sigma_i||)$.
	So
	\begin{align*}
		f_jf_l\iv||\sigma_i||^{\alpha_j+\alpha_k-2\alpha_l}\overline{b}^j_l\wedge b^k_l=f_jf_l\iv(||\sigma_i||^{\alpha_j-\alpha_l}\overline{b}^j_l)\wedge (||\sigma_i||^{\alpha_k-\alpha_l}b^k_l)=O(1).
	\end{align*}
	Now
	\begin{align*}
		\tilde{F}_j^k=&\left[\dd b_j^k-(\alpha_j-\alpha_k)z_1\iv dz_1\wedge b_j^k\right]||\sigma_i||^{\alpha_k-\alpha_j}+(\alpha_j-\alpha_k)\overline{b}^j_k\wedge\bar{z}_1\iv d\bar{z}_1||\sigma_i||^{\alpha_j-\alpha_k}f_jf_k\iv+O(1).
	\end{align*}
	This proves the proposition.
	 \end{proof}
Now we consider the case that $\mbar$ is a compact surface. We will choose a special background metric $H_1$ and special extensions $F^i_j$ of $\fb^i_j$ such that $F_{\hp}\in L^p$ for some $p>2$ and $|\Lambda_\omega F_{\hp}|=O(1)$. The construction is a generalization of that in \cite{MR1373062}.

First, we can assume $\eb|_{U_i\cap U_j}$ is trivial after possibly shrinking $U$ since $D_i\cap D_j$ is just a set of finite points. Given holomorphic functions $f_1(z_1)$ and $f_2(z_2)$ with $f_1(0)=f_2(0)$, there exists a holomorphic function $f(z_1,z_2)$ such that $f(z_1,0)=f_1(z_1)$ and $f(0,z_2)=f_2(z_2)$ by simply taking $f(z_1,z_2):=f_1(z_1)+f_2(z_2)-f_1(0)$. Thus on $U_i\cap U_j$, we can extend $\fb^i_k$ and $\fb^j_k$ holomorphically, i.e. there exist holomorphic vector bundles $\fb^{ij}_k$ over $U_i\cap U_j$ such that $\fb^{ij}_k|_{D_i\cap U_j}=\fb^i_k$, $\fb^{ij}_k|_{D_j\cap U_i}=\fb^j_k$. Now we have a flag of trivial holomorphic vector bundles:
\begin{align*}
	\eb|_{U_i\cap U_j}=\fb^{ij}_1\supseteq \fb^{ij}_2\supseteq\cdots \supseteq \fb^{ij}_r\supseteq0.
\end{align*}
Let $\{e^{ij}_k\}$ be the holomorphic frame such that $\fb^{ij}_k=\spann\{e^{ij}_k,\cdots,e^{ij}_r\}$. Now we can define $H_1|_{U_i\cap U_j}$ such that $H_1(e^{ij}_k,e^{ij}_l):=\delta_{kl}$. Note that $F(H_1|_{U_i\cap U_j})=0$.

Next, we extend $H_1|_{U_i\cap U_j}$ smoothly on the whole $\mbar$ to get Hermitian metric $H_2$. By \Cref{lmm-foliation}, there exists a smooth foliation of the tubular neighborhood of $D_i$. For any $x\in D_i$, let $B_x$ be the leaf of $x$, which is a complex disc and orthogonal to $D_i$. By \cite{MR1165874}, we can solve the Dirichlet problem
\begin{align*}
	\begin{cases}
		F(H_1|_{B_x})=0,\\
		H_1|_{\partial B_x}=H_2.
	\end{cases}
\end{align*}
The solution always exists and is unique. Due to uniqueness, $H_1$ coincides with the metric we constructed on $U_i\cap U_j$.

Let $\nabla_1$ be the Chern connection of $H_1$, we can extend $\fb^i_j$ smoothly to $F^i_j$ by parallel transport along $B_x$ with respect to $\nabla_1$. On $U_i\cap U_j$, since $\nabla_1 e^{ij}_k=0$, we have $F^i_k|_{U_i\cap U_j}=\fb^{ij}_k=F^j_k|_{U_i\cap U_j}$. Then $Q^i_k|_{U_i\cap U_j}=Q^j_k|_{U_i\cap U_j}$. So $\hp$ is well defined.
\begin{lemma}\label{lem-Fbdd-local}
	Let $\hp$ be the metric constructed above. Then $\left|F_{\hp}\right|_{\hp}$ is bounded on $U_i\cap U_j$.
\end{lemma}
\begin{proof}
	Since $H_1(e^{ij}_k,e^{ij}_l)=\delta_{ij}$, we have
	\begin{align*}
		\hp(e^{ij}_k,e^{ij}_l)=\delta_{kl}||\sigma_1||^{2\alpha^1_k}\cdots||\sigma_m||^{2\alpha^m_k}.
	\end{align*}
	Then
	\begin{align*}
		(F_{\hp})_k^l=\db(\dd\hp \hp\iv)_k^l=\delta_k^l(\alpha^i_k\db\dd\log||\sigma_i||^2+\alpha^j_k\db\dd\log||\sigma_j||^2+\db\dd\log c_k),
	\end{align*}
	where $c_k=\prod_{s\neq i,j}||\sigma_s||^{2\alpha^s_k}$.
	
	On $U_i\cap U_j$, we can assume $||\sigma_i||^2=f_1|z_1|^2$, $||\sigma_j||^2=f_2|z_2|^2$. Thus
	\begin{align*}
		(F_{\hp})_k^l=\delta_k^l(\alpha^i_k\db\dd\log f_1+\alpha^j_k\db\dd\log f_2+\db\dd\log c_k)=O(1).
	\end{align*} 
\end{proof}
\begin{proposition}\label{prop-tracecurvaturebdd}
	Let $\hp$ be the metric constructed above. Then $\big|\Lambda_\omega F_{\hp}|_{\hp}$ is bounded on $M$.
\end{proposition}
\begin{proof}
	By Lemma \ref{lem-Fbdd-local} and Proposition \ref{prop-Fcurvature-compute}, it suffices to check the assumptions made in Proposition \ref{prop-Fcurvature-compute}. For any $\alpha^i_j>\alpha^i_k$, there exists $\fb^i_s$ such that $e_j|_{D_i}\in \fb^i_s$ but $e_k|_{D_i}\notin \fb^i_s$. For any $x\in D\cap V$, $T^{0,1}_x\mbar=T^{0,1}_xD\oplus N^{0,1}_xD$, where $N_xD$ is the orthogonal complement of $T_xD$ with respect to $\omega$. Since $b_j^k=H_1(\db_\eb e_j,e_k)\in A^{0,1}(V)$, we have
	\begin{align*}
		\langle b_j^k|_D,T^{0,1}_xD\rangle&=H_1(\langle\db_\eb e_j|_D,T^{0,1}_xD\rangle,e_k)=H_1(\langle\db_{\fb^i_s} e_j|_D,T^{0,1}_xD\rangle,e_k)=0,\\
		\langle b_j^k|_D,N^{0,1}_xD\rangle&=H_1(\langle\db_\eb e_j|_D,N^{0,1}_xD\rangle,e_k)=H_1(\langle\nabla e_j|_D,N^{0,1}_xD\rangle,e_k)=0,
	\end{align*}
	where the last equality follows from $e_j$ is extended by parallel transport. So we have $b_j^k|_D=0$ for any $\alpha^i_j>\alpha^i_k$.

    For any $p\in D_i$, take a coordinate $(U,z_1,z_2)$ such that $\sigma_i=z_1$. By the construction of the foliation $\{B_x\}$ orthogonal to $D_i$ in \Cref{lmm-foliation}, there exists functions $\{f_a(z_1):a\in D_i\cap U, |z_1|<1\}$ satisfying
    \begin{itemize}
        \item $B_a=\{(z_1,z_2)\in U:z_2=f_a(z_1)\}$. In particular, $f_a(0)=a$.
        \item $f_a$ is holomorphic w.r.t $z_1$,
        \item $f_a$ is smooth w.r.t. $a$.
    \end{itemize}
    We can define a function $a(z_1,z_2):U\to\cpxn$ by taking $a|_{B_a}\equiv a$. Then $a(0,z_2)=z_2$. By implicit function theorem, we have $da|_D=-f'_a(0)dz_1+dz_2\in T^*_{1,0}$. Thus $\db a=O(||\sigma_i||)$. Since in the coordinate $(z_1,a)$, $\langle\frac{\dd}{\dd z_1},\frac{\dd}{\dd a}\rangle|_D=0$, we have
    \begin{align*}
        \omega=\ii\left(g_{1\bar1}dz_1\wedge d\bar{z}_1+g_{2\bar2}da\wedge d\bar{a}\right)+O(||\sigma_i||)
    \end{align*}
     for some smooth functions $g_{1\bar1},g_{2\bar2}$.  

    Assume $b_j^k=\lambda d\bar{z}_1+\mu d\bar{z}_2$ in the coordinate $(z_1,z_2)$, and we know that $\lambda|_D=\mu|_D=0$.
    \begin{align*}
        \Lambda_\omega\dd b_j^k|_D&=\left.g^{1\bar1}\dd b_j^k\left(\dd_1+f'_a(0)\dd_2,\overline{\dd_1+f_a'(0)\dd_2}\right)+g^{2\bar2}\dd b_j^k(\dd_2,\dd_{\bar2})\right|_D\\
        &=\left.g^{1\bar1}\left(\dd_1\lambda+f'_a(0)\dd_2\lambda+\bar{f}'_a(0)\dd_1\mu+|f'_a(0)|^2\dd_2\mu\right)+g^{2\bar2}\dd_2\mu\right|_D,
    \end{align*}
    where $\dd_i=\frac{\dd}{\dd z_i}$. By $\lambda|_D=\mu|_D=0$, we have $\dd_2\lambda|_D=\dd_2\mu|_D=0$. Since $\{e_i\}$ are constructed by parallel transport along $B_x$ w.r.t a flat connection, we have $\langle\nabla e_i,TB_x\rangle=0$. In particular, $\langle b_j^k,T^{0,1}B_x\rangle\equiv0$, i.e. $\lambda+\bar{f}'_a(z_1)\mu\equiv0$. Differentiate by $z_1$ and note that $\mu|_D=0$, we have $\dd_1(\lambda+\bar{f}'_a(0)\mu)|_D=\dd_{\bar1}(\lambda+\bar{f}'_a(0)\mu)|_D=0$. So we have
    \begin{align*}
        \Lambda_\omega\dd b_j^k|_D=\left.g^{1\bar1}\left(\dd_1\lambda+\bar{f}'_a(0)\dd_1\mu\right)\right|_D=0.
    \end{align*}

    Next, $d\sigma_i\wedge b_j^k=\lambda dz_1\wedge d\bar{z}_1+\mu dz_1\wedge d\bar{z}_2$. Assume $\dd_1+v\dd_2$ is orthogonal to $\dd_2$ for some $v\in C^\infty(U)$, then $v=f'_a(0)+O(||\sigma_i||)$. Now we have
    \begin{align*}
        \Lambda_\omega(d\sigma_i\wedge b_j^k)&=g^{1\bar1}(d\sigma_i\wedge b_j^k)\left(\dd_1+v\dd_2,\overline{\dd_1+v\dd_2}\right)+g^{2\bar2}(d\sigma_i\wedge b_j^k)(\dd_2,\dd_{\bar2})\\
        &=\lambda+\bar{v}\mu.
    \end{align*}
    Since $(\lambda+\bar{f}'_a(0)\mu)|_D=\dd_1(\lambda+\bar{f}'_a(0)\mu)|_D=\dd_{\bar1}(\lambda+\bar{f}'_a(0)\mu)|_D=0$, we have $\lambda+\bar{f}'_a(0)\mu=O(||\sigma_i||^2)$. Note $\mu=O(||\sigma_i||)$, we have
    \begin{align*}
        \Lambda_\omega(d\sigma_i\wedge b_j^k)=\lambda+\bar{v}\mu=\lambda+\bar{f}'_a(0)\mu+O(||\sigma_i||)\mu=O(||\sigma_i||^2).
    \end{align*}

    Finally, we have
    \begin{align*}
        \Lambda_\omega(d\sigma_i\wedge b_k^j)|_D=g^{1\bar1}\left.\left<\bar{b}^j_k,T^{1,0}B_x\right>\right|_D=0,
    \end{align*}
    which follows from the parallel transport along $B_x$.
\end{proof}

\begin{definition}
    A Hermitian metric $H$ on $\eb|_{M}$ is called \emph{compatible with the parabolic structure} if $H=\hp h$ for $h\in \Gamma(\End(\eb|_M))$ satisfies the following:
    \begin{enumerate}
        \item $H,\hp$ are mutually bounded, i.e. there exists $C>0$ such that for any $x\in M$, $v_x\in \eb_x$, we have $C^{-1}|v_x|_{\hp}\leqs |v_x|_{H}\leqs C|v_x|_{\hp}$.
        \item $|\db h|_{\hp,\omega}\in L^2(M,\omega)$.
         \item $|\Lambda_\omega F_H|_H\in L^1(M,\omega)$.
    \end{enumerate}
\end{definition}

\begin{definition}
    Let $\eb$ be a holomorphic vector bundle over $(M,\omega)$, a Hermitian metric $H$ is called
    \emph{Hermitian-Einstein} with respect to $\omega$ if there exists a constant $\lambda\in\rean$ such that
    \begin{align*}
        \ii\Lambda_{\omega}F_H=\lambda\cdot \id_\eb.
    \end{align*} 
\end{definition}
\begin{definition}
    Let $\eb$ be a holomorphic vector bundle. $\eb$ is called \emph{decomposable} if there exist proper holomorphic subbundles $\vb_1$, $\vb_2$ such that $\eb\simeq\vb_1\oplus\vb_2$. $\eb$ is called \emph{indecomposable} if $\eb$ is not decomposable.
\end{definition}

Now we are ready to state the main theorem of this paper.


\begin{theorem}\label{thm-KHset2}
    Let $(\mbar,\omega)$ be a compact complex surface with a Gauduchon metric. $D=\sum_{i=1}^m D_i$ be a simple normal crossing divisor.
    $\eb$ is a holomorphic vector bundle over $\mbar$ with a parabolic structure.

    If $\eb$ is parabolic stable, then there exists a Hermitian-Einstein metric on $\eb|_M$ compatible with
    the parabolic structure with respect to $\omega$.

    If $\eb|_M$ admits a Hermitian-Einstein metric compatible with the parabolic structure with respect to $\omega$ and $\eb$ is indecomposable,
    then $\eb$ is parabolic stable. 
\end{theorem} 


\section{Analytic stabilty}
The proof of similar results in \cite{MR1701135,MR1775134,MR2310103} are all based on Simpson's work\cite{MR944577} on existence of Hermitian-Einstein metric of Higgs bundle over noncompact manifolds. More precisely, Simpson considered the noncompact \kah manifolds with the following three additional assumptions (A1)-(A3):
\begin{enumerate}
    \item[(A1).] $(M,g)$ has finite volume.
    \item[(A2).] There exists a non-negative exhaustion function $\phi$ with $\ii\Lambda_\omega\dd\db\phi$ bounded.
    \item[(A3).] There exists an increasing function $a:[0,\infty)\to[0,\infty)$ with $a(0)=0$ and $a(x)=x$ for $x>1$, such that if $f$ is a bounded positive function on $M$ with $\ii\Lambda_\omega\dd\db f\geqs-B$, then
    \begin{align*}
        \sup_M|f|\leqs C_B a\left(\int_M|f|\frac{\omega^n}{n!}\right).
    \end{align*}
    Futhermore, if $B=0$, then $\ii\Lambda_\omega\dd\db f=0$.
\end{enumerate}

For similar reason, our proof is based on the following theorem, which is a non\kah generalization of Simpson's work.
\begin{theorem}[\cite{MR4237961},Theorem 1.1]\label{thm-Zhangxi} Let $(M,g)$ be a noncompact Gauduchon manifold satisfying the following Assumptions 1,2,3 and $|d\omega^{n-1}|_g\in L^2(M)$, $(E,\db_E,\theta)$ be a Higgs bundle with a Hermitian metric $K$ satisfying $\sup_M|\Lambda_\omega F_{K,\theta}|_K<+\infty$ over $M$. If $(E,\db_E,\theta)$ is $K$-analytically stable (see \Cref{def-analyticstable}), then there exists a Hermitian metric $H$ with $\db_\theta(\log K\iv H)\in L^2(M)$, $H$ and $K$ mutually bounded, such that
\begin{align*}
    \ii\Lambda_\omega(F_H+[\theta,\theta^{*H}])=\lambda_{K,\omega}\cdot \mathrm{Id}_E.
\end{align*}
    
\end{theorem}
Since the metric $\omega$ in our case is just the restriction metric $\omega|_M$, the extra assumption $|d\omega^{n-1}|_g\in L^2(M)$ in non\kah case obviously holds. For the use of the theorem above, we need to give the definition of analytic degree and show the equivalence of parabolic stablity and analytic stability of metrics compatible with the parabolic structure.


\begin{definition}
    Suppose $H$ is a Hermitian metric on $\eb$ over $M$. The \emph{analytic degree} $d(\eb,H)$ of $(\eb, H)$ with respect to $\omega$ is defined by
    \begin{align*}
        d(\eb,H):=\frac{\ii}{2\pi}\int_M\tr F_H\wedge\omega.
    \end{align*}
\end{definition}
\begin{remark}
    For the Hermitian metric on $\eb$ over $\mbar$, the analytic degree is exactly the classical degree defined before.
\end{remark}
\begin{proposition}\label{prop-e-an=par}
    For any Hermitian metric $H$ compatible with the parabolic structure, we have $$d(\eb,H)=\pardeg_{\omega}\eb.$$
\end{proposition}
\begin{proof}
    Let $K_1$ be a smooth Hermitian metric on $\eb$ over $\mbar$, and $K_1\iv H_{par}=h\in \End(\eb|_M)$. Recall that on the tubular neighborhood $U_i$ of $D_i$, we have
    \begin{align*}
        \hp|_{U_i\backslash D}=\bigoplus_j\left(\prod_i||\sigma_i||^{2\alpha^i_j}\right)H_1|_{Q^i_j},
    \end{align*}
    where $H_1$ is smooth on $U_i$. Thus
    \begin{align*}
        \det h|_{U_i\backslash D}=\det K_1\iv\hp|_{U_i\backslash D}=\left(\prod_{i,j}||\sigma_i||^{2k^i_j\alpha^i_j}\right)\det(K_1\iv H_1),
    \end{align*}
    where $k^i_j=\rank(\fb^i_j/\fb^i_{j+1})$.
    So $f:=\left(\prod_{i,j}||\sigma_i||^{-2k^i_j\alpha^i_j}\right)\det h\in C^\infty(M)$ can be extended smoothly to $C^\infty(\mbar)$.

    First we show that $d(\eb,H_{par})=\pardeg_{\omega}\eb$.
    \begin{align*}
        d(\eb,H_{par})&=\frac{\ii}{2\pi}\int_M\tr F_{\hp}\wedge\omega\\
        &=\frac{\ii}{2\pi}\int_M\tr F_{K_1}\wedge\omega+\frac{\ii}{2\pi}\int_M\db\dd\log\det h\wedge\omega\\
        &=\deg_\omega\eb+\frac{\ii}{2\pi}\int_M\db\dd\log\det h\wedge\omega\\
        &=\deg_\omega\eb+\frac{\ii}{2\pi}\int_M\db\dd\log\left(\prod_{i,j}||\sigma_i||^{2k^i_j\alpha^i_j}\right)\wedge\omega+\frac{\ii}{2\pi}\int_M\db\dd\log f\wedge\omega,
    \end{align*}
    for some $f\in C^\infty(\mbar,\rean_+)$ by the previous claim. Since $\dd\db\omega=0$, we have
    \begin{align*}
        \int_M\db\dd\log f\wedge\omega=\int_{\mbar}\db\dd\log f\wedge\omega=\int_{\mbar}\log f \dd\db\omega=0.
    \end{align*}
    By the Poincare-Lelong formula (Proposition \ref{prop-poincare-lelong}), 
    \begin{align*}
        \frac{\ii}{2\pi}\db\dd\log||\sigma_i||^2=c_1^{BC}([D_i])-\delta_{D_i},
    \end{align*}
    we have
    \begin{align*}
        \frac{\ii}{2\pi}\int_M\db\dd\log\left(\prod_{i,j}||\sigma_i||^{2k^i_j\alpha^i_j}\right)\wedge\omega&=\frac{\ii}{2\pi}\sum_{i,j}k^i_j\alpha^i_j\int_M\db\dd\log||\sigma_i||^2\wedge\omega\\
        &=\sum_{i,j}\rank(\fb^i_j/\fb^i_{j+1})\alpha^i_j\deg_\omega[D_i].
    \end{align*}
    Thus we prove the proposition for $H=\hp$.


    Next we show that $d(\eb,H)=\pardeg_{\omega}\eb$ for general metric $H$ compatible with the parabolic structure. Set $H=\hp h$ for $h\in \Gamma(\End(\eb|_M))$. We have $\tr F_H=\tr F_{\hp}-\dd\db\log\det h$.
    \begin{align*}
        d(\eb,H)&=\frac{\ii}{2\pi}\int_M\tr F_{H}\wedge\omega\\
        &=d(\eb,\hp)-\frac{\ii}{2\pi}\int_M \dd\db\log\det h\wedge\omega.
    \end{align*}
    By the compatible condition, we have
    \begin{align*}
        &|\db\log\det h|=|\tr(\db h h\iv)|\leqs C|\db h|_{\hp,\omega}\in L^2(M,\omega),\\
        &\Lambda_\omega\dd\db\log\det h=\tr \Lambda_\omega F_{\hp}-\tr \Lambda_\omega F_H\in L^1(M,\omega),\\
        &\log\det h\in L^2(M,\omega).
    \end{align*}
    So by \cite[Lemma 5.2]{MR944577}, we have
    \begin{align*}
        \int_M \dd\db\log\det h\wedge\omega=\int_M \db\log\det h\wedge\dd\omega=-\int_M \log\det h\cdot \db\dd\omega=0.
    \end{align*}
\end{proof}

\begin{definition}\label{def-analyticstable}
    Let $\vb$ be a proper coherent subsheaf of $\eb|_M$ with torsion free quotient, $H$ be a Hermitian metric on $\eb$ over $M$. The \emph{analytic degree} $d(\vb,H)$ of $(\vb,H)$ with respect to $\omega$ is defined by
    \begin{align*}
        d(\vb,H):=\frac{\ii}{2\pi}\int_{M}\tr(F_{H|_\vb})\wedge \omega.
    \end{align*} 
    We call $(\eb,H)$ is \emph{analytic stable} with respect to $\omega$ if for any proper coherent subsheaf $\vb$ of $\eb|_M$ with torsion free quotient, we always have
    \begin{align*}
        \frac{d(\vb,H)}{\rank\vb}<\frac{d(\eb,H)}{\rank\eb}.
    \end{align*}
\end{definition}
\begin{remark}
    Since $\vb$ may be not a subbundle of $\eb$, $H|_\vb$ is not always a Hermitian metric on $\vb$. So it's not obvious that the analytic degree is equal to the classical degree for the metric $H$ on $\eb$ over $\mbar$.  However, we can still prove the following proposition by blowing-up the singularities.
\end{remark}
\begin{proposition}\label{prop-analyticdegree-subsheaf-classical}
    Suppose $\vb$ is a proper subsheaf with torsion free quotient. Let $K_1$ be a Hermitian metric on $\eb$ over $\mbar$. Then we have
    \begin{align*}
        d(\vb,K_1)=\deg_\omega \vb.
    \end{align*}
\end{proposition}
\begin{proof}
     The subsheaf $\vb\subset\eb$ gives $\wedge^s\vb\subset\wedge^s\eb$, where $s$ is the rank of $\vb$. Since $\deg_\omega\vb=\deg_\omega(\wedge^s\vb)$ and $d(\vb,K_1)=d(\wedge^s\vb,K_1)$, we can assume $\vb$ is of rank $1$. By Proposition \ref{prop-sheafregularity}, $\vb$ is locally free outside a subset of codim$\geqs3$. Since $\mbar$ is a complex surface, $\vb$ is actually a holomorphic vector bundle. The quotient sheaf $\eb/\vb$ is locally free outside a subset of codim$\geqs2$, i.e. a set of isolated points. For each singular point $x_0$, we can blow up $x_0$ to get $\pi:\widetilde{M}:=\mathrm{Bl}_{x_0}\mbar\to\mbar$ and the exceptional divisor $P_0:=\pi\iv(x_0)$. The exact sequence $0\to\oo_{\mbar}\to\eb\otimes\vb^\vee\to(\eb/\vb)\otimes \vb^\vee\to0$ naturally induces a long exact sequence
    \begin{align*}
        0 \to H^0(\mbar,\oo_{\mbar}) \xrightarrow{i_*} H^0(\mbar,\eb\otimes\vb^\vee).
    \end{align*}
    Since $H^0(\mbar,\oo_{\mbar})\simeq\cpxn$, we get a global section $s:=i_*(1)\in H^0(\mbar,\eb\otimes\vb^\vee)$. In fact, the set of singular points $\{x_0\}$ is exactly the vanishing set of $s$. Take $\tilde{s}:=\pi^*(s)\in H^0(\widetilde{M},\pi^*(\eb)\otimes\pi^*(\vb)^\vee)$. Let $\tilde{s}=(s_1,s_2,\cdots,s_r)$ near $P_0$ with respect to a trivialization of $\pi^*(\eb)\otimes\pi^*(\vb)^\vee$. Suppose $m=\min\{\ord_{P_0}s_i:i=1,\cdots, r\}>0$ to be minimum of the vanishing order, then $\tilde{s}\in H^0(\widetilde{M},\pi^*(\eb)\otimes\pi^*(\vb)^\vee\otimes\oo_{\widetilde{M}}(-m P_0))$ is a nowhere vanishing global section. Hence $\pi^*(\vb)$ is a line subbundle of $\pi^*(\eb)\otimes\oo_{\widetilde{M}}(-m P_0)$.

    Let $\rho\in H^0(\widetilde{M},\oo_{\widetilde{M}}(P_0))$ be the canonical section and $h$ be a Hermitian metric on $\oo_{\widetilde{M}}(P_0)$. Then $\gamma:=(\pi^*(K_1)\otimes h)|_{\pi^*\vb}$ is a Hermitian metric on $\pi^*\vb$. Let $\gamma_0$ be a Hermitian metric on $\vb$, then by definition
    \begin{align*}
        \deg_\omega\vb&=\frac{\ii}{2\pi}\int_{\mbar}\tr(F_{\gamma_0})\wedge \omega\\
        &=\frac{\ii}{2\pi}\int_{\mbar\backslash\{x_0\}}\tr(F_{\gamma_0})\wedge \omega\\
        &=\frac{\ii}{2\pi}\int_{\widetilde{M}\backslash P_0}\tr(F_{\pi^*\gamma_0})\wedge \pi^*\omega\\
        &=\frac{\ii}{2\pi}\int_{\widetilde{M}\backslash P_0}\tr(F_{\gamma})\wedge \pi^*\omega.
    \end{align*}
    The last equality follows from the fact that $\pi^*\gamma_0$ and $\gamma$ are both Hermitian metrics on $\pi^*\vb|_{\widetilde{M}\backslash P_0}$, so they differs a $\dd\db$ exact form.
    \begin{align*}
        \deg_\omega\vb&=\frac{\ii}{2\pi}\int_{\widetilde{M}\backslash P_0}\tr(F_{\gamma})\wedge \pi^*\omega\\
        &=\frac{\ii}{2\pi}\int_{\widetilde{M}\backslash P_0}\tr(\pi^*F_{K_1|_\vb})\wedge \pi^*\omega+\frac{\ii}{2\pi}\int_{\widetilde{M}\backslash P_0}\left(-m \db\dd\log||\rho||_h\right)\wedge \pi^*\omega\\
        &=\frac{\ii}{2\pi}\int_{M\backslash \{x_0\}}\tr(F_{K_1|_\vb})\wedge \omega-\frac{m}{2}\int_{\widetilde{M}\backslash P_0}c_1^{BC}(\oo_{\tilde{M}}(P_0))\wedge \pi^*\omega\\
        &=\frac{\ii}{2\pi}\int_{M}\tr(F_{K_1|_\vb})\wedge \omega,
    \end{align*}
    where the last two inequalities follow from the Poincar\'{e}-Lelong formula (\Cref{prop-poincare-lelong}) and \Cref{prop-projection}.
\end{proof}
\begin{proposition}\label{prop-subsheaf-andeg=pardeg}
    Suppose $\eb$ is holomorphic vector bundle over $\mbar$, and $\vb$ is a coherent subsheaf of $\eb$ with torsion free quotient. For the Hermitian metric $H_{par}$ constructed in previous section, we have
    \begin{align*}
        d(\vb,H_{par})=\pardeg_\omega\vb.
    \end{align*}
\end{proposition}
\begin{proof}
    Since $\mbar$ is a complex surface, $\vb$ is a holomorphic vector bundle and subbundle of $\eb$ outside finite points. As the proof of \Cref{prop-analyticdegree-subsheaf-classical}, we blow up the singular point $x_0$ of the quotient sheaf $\eb/\vb$. Let $D_i^*$ be the strict transform of $D_i$, and the total transform $\widetilde{D}=\sum_{i=1}^{m}D_i^*+P_0$ is a simple normal crossing divisor. Suppose $\sigma_i$, $\sigma_i^*$, $\rho$ are the canonical section of $\oo_{\mbar}(D_i)$, $\oo_{\widetilde{M}}(D_i^*)$ and $\oo_{\widetilde{M}}(P_0)$ respectively, then we have $\pi^*(\sigma_i)=\sigma_i^*\otimes\rho$ if $x_0\in D_i$. 

    Recall the parabolic degree of $\vb$ is defined by
    \begin{align*}
        \pardeg_\omega\vb:=\deg_\omega\vb+\sum_{i,j}k^i_j\beta^i_j\deg_\omega[D_i],
    \end{align*}
    where $\beta^i_j$ are weights attached to the induced flag $\vb|_{D_i}=\fb^i_1\vb\supseteq \cdots\supseteq \fb^i_{n_i}\vb\supseteq 0$ and $k^i_j:=\rank(\fb^i_j\vb/\fb^i_{j+1}\vb)$ are multiplicities of $\beta^i_j$.

    Since we have already shown that
    \begin{align*}
        \frac{\ii}{2\pi}\int_{M}\tr(F_{K_1|_\vb})\wedge \omega=\deg_\omega\vb,
    \end{align*}
    and
    \begin{align*}
        &\frac{\ii}{2\pi}\int_{\widetilde{M}\backslash\pi^*(D)}\db\dd\log\left(\prod_{i,j}||\pi^*(\sigma_i)||^{2k^i_j\beta^i_j}\right)\wedge\pi^*\omega\\
        =& \frac{\ii}{2\pi}\int_{M}\db\dd\log\left(\prod_{i,j}||\sigma_i||^{2k^i_j\beta^i_j}\right)\wedge\omega\\
        =& \sum_{i,j}k^i_j\beta^i_j\deg_\omega[D_i],
    \end{align*}
    it suffices to show
    \begin{align*}
        \frac{\ii}{2\pi}\int_{\widetilde{M}\backslash\pi^*(D)}\db\dd\log\left(\prod_{i,j}||\pi^*(\sigma_i)||^{-2k^i_j\beta^i_j}\det(\pi^*K_1|_{\pi^*\vb})^{-1}(\pi^*\hp|_{\pi^*\vb})\right)\wedge\pi^*\omega=0.
    \end{align*}

    Let $\Psi:=\left(\prod_{i,j}||\pi^*(\sigma_i)||^{-2k^i_j\beta^i_j}\right)\det\left((\pi^*K_1|_{\pi^*\vb})^{-1}(\pi^*\hp|_{\pi^*\vb})\right) \in C^\infty(\widetilde{M}\backslash\pi^*(D))$. By \cite[Lemma 5.2]{MR944577}, it suffices to check the following integrability conditions:
    \begin{align*}
        \db\dd\log\Psi\wedge\pi^*\omega\in L^1,\ \dd\log\Psi\in L^2,\ \log\Psi\in L^2.
    \end{align*}
    Under these conditions, we have
    \begin{align*}
        \int\db\dd\log\Psi\wedge\pi^*\omega=\int \dd\log\Psi\wedge\pi^*\db\omega=-\int\log\Psi\pi^*\dd\db\omega=0.
    \end{align*}

    First, $\log\Psi\in L^2$ follows from the fact that $\hp$ has at most polynomial growth. By the construction, there exists constant $C>0$, $\lambda_i>0$, such that
    \begin{align*}
        C\iv\prod_i||\pi^*\sigma_i||^{\lambda_i} \leqs\Psi\leqs C\prod_i||\pi^*\sigma_i||^{-\lambda_i}.
    \end{align*}
     Thus $\left|\log\Psi\right|\leqs C(\sum_i\log||\pi^*\sigma_i||+1)$. Since $\log|z|\in L^2(B_1)$ and $\log(|z_1|^2+|z_2|^2)\in L^2(B_1)$, we have $\log\Psi\in L^2$.


    For any $p\in\pi^*(D_i)\backslash P_0$, we take a neighborhood $U$ of $p$. For simplicity of notation, we assume each weight is of multiplicity $1$, and general case is similar. Let $f_1,\cdots,f_s$ be the holomorphic basis of $\vb$ on $U$, such that $\fb^i_j\vb=\langle f_j,\cdots,f_s\rangle$. Let $e^i_k$ be the smooth basis of $Q^i_j$ constructed in previous section, then there exists functions $c_{kj}\in C^\infty(U)$ such that 
    \begin{align*}
        f_k=c_{kj}\cdot\pi^*e^i_j\otimes f,
    \end{align*}
    where $f$ is holomorphic basis of $\oo(-P_0)$. Assume $\beta^i_k=\alpha^i_{j_k}$, we have $c_{kj}\neq 0$ for $j=j_k$ and $c_{kj}=0$ for $j<j_k$. By definition of $\hp$, we have
    \begin{align*}
        \langle f_k,f_l\rangle_{\hp}=\sum_{j=1}^r\prod_{i=1}^m||\sigma_i||^{2\alpha^i_j}c_{kj}\bar{c}_{lj}.
    \end{align*}
    Let $K_0:=\left(\langle f_k,f_l\rangle_{\hp}\right)_{1\leqs k,l\leqs s}$, $S:=\diag\left\{\prod_i||\sigma_i||^{\alpha^i_1},\cdots,\prod_i||\sigma_i||^{\alpha^i_r}\right\}$ and $C:=(c_{kj})_{1\leqs k\leqs s,1\leqs j\leqs r}$, then we have $K_0=CS^2C^*=(CS)(CS)^*$, where $(CS)_{kj}=c_{kj}\prod_i||\sigma_i||^{\alpha^i_j}$. By Cauchy-Binet formula, we have
    \begin{align*}
        \prod_{i,j}||\sigma_i||^{-2\beta^i_j}\det K_0=\sum_{1\leqs \tilde{j}_1<\cdots<\tilde{j}_s\leqs r}\left|\det\left(c_{k\tilde{j}_l}\prod_i||\sigma_i||^{\alpha^i_{\tilde{j}_l}-\alpha^i_{j_k}}\right)\right|^2=A+\sum_l B_l\prod_i||\sigma_i||^{2\gamma_{il}},
    \end{align*}
    where $\log A,\log B_l\in C^\infty(U)$, $\gamma\geqs \min\{\alpha^i_{j+1}-\alpha^i_j:1\leqs j\leqs r\}$. Moreover, since $K_1$ is smooth metric on $\mbar$, $\Psi$ has same form as above. Thus,
    \begin{align*}
        \db\dd\log\Psi=&\sum_i A_{i}\dd\db||\sigma_i||^{2\gamma_i}+\sum_{i\neq j}B_{ij}\dd||\sigma_i||^{2\gamma_i}\db||\sigma_j||^{2\gamma_j}+l.o.t=O(||\sigma_i||^{2\gamma-2})\in L^1,\\
        \dd\log\Psi=&\sum_i C_i\dd||\sigma_i||^{2\gamma_i}+O(1)=O(||\sigma_i||^{2\gamma-1})\in L^2.
    \end{align*}

    If $p\in P_0$, we first compute the expression of $\pi^*\omega$. Recall the blowup is given by
    \begin{align*}
        \Bl_{x_0}U:=\{\left((z_1,z_2),[\xi_0:\xi_1]\right)\in U\times\proj^1:z_1\xi_1=z_2\xi_0\}.
    \end{align*}
    In the affine chart $\xi_0=1$, we have $z_2=z_1\xi_1$. The exceptional divisor $P_0$ is given by $z_1=0$. Then 
    \begin{align*}
        \pi^*\omega&=\pi^*\left(\ii\sum_{i,j=1}^2g_{i\bar{j}}dz_i\wedge d\bar{z}_j\right)\\
        &=a_{1\bar{1}}dz_1\wedge d\bar{z}_1+a_{1\bar{2}}\bar{z}_1 dz_1\wedge d\bar{\xi}_1+a_{2\bar{1}}z_1d\xi_1\wedge d\bar{z}_1+a_{2\bar{2}}|z_1|^2d\xi_1\wedge d\bar{\xi}_1.
    \end{align*}
    Since we don't have the vanishing of some $c_{kj}$ on $P_0$, in general we only have $\Psi=\sum_i A_i||z_1||^{2\lambda_i}$ for some $\lambda_i\in\rean$ and $\log A_i\in C^\infty(U)$. Some of $\lambda_i$ may be negative. Let $\lambda_0:=\min\{\lambda_i\}$, then $\log\Psi=\lambda_0\log|z_1|^2+\log\left(A_0+\sum_{i\neq0} A_i|z_1|^{2\tilde{\lambda}_i}\right)$, where $\tilde{\lambda}_i=\lambda_i-\lambda_0>0$ for $i\neq0$. Note that
    \begin{align*}
        \dd\db\log\left(A_0+\sum_{i\neq0} A_i|z_1|^{2\tilde{\lambda}_i}\right)&=O(|z_1|^{2\tilde{\lambda}-2})\in L^1,\\
        \dd\log\left(A_0+\sum_{i\neq0} A_i|z_1|^{2\tilde{\lambda}_i}\right)&=O(|z_1|^{2\tilde{\lambda}-1})\in L^2.
    \end{align*}
    By Poincar\'{e}-Lelong formula, we have $\dd\db(\lambda_0\log|z_1|^2)\in L^{\infty}$. Thus $\dd\db\log\Psi\in L^1$. However, $\dd\log|z_1|^2=z_1\iv dz_1\notin L^2$. Actually, to integrate by part, it suffices to get $\dd\log\Psi\wedge\pi^*\omega\in L^2$. By the computation above,
    \begin{align*}
        \dd\log|z_1|^2\wedge\pi^*\omega=a_{2\bar{1}}dz_1\wedge d\xi_1\wedge d\bar{z}_1+a_{2\bar{2}}\bar{z}_1dz_1\wedge d\xi_1\wedge d\bar{\xi}_1\in L^\infty.
    \end{align*}
    Thus we always have $\dd\log\Psi\wedge\pi^*\omega\in L^2$. This proves the proposition.

\end{proof}

\begin{proposition}\label{prop-subsheaf-an=par}
    For any Hermitian metric $H$ compatible with the parabolic structure, we have 
    \begin{align*}
        d(\vb,H)=\pardeg_{\omega}\vb.
    \end{align*}
\end{proposition}
\begin{proof}
    Let $h=\hp\iv H\in \End(\eb|_M)$. Let $e^i_j$ be the basis defined above, then $\tilde{e}^i_j:=\left(\prod_t||\sigma_t||^{-\alpha^t_j}\right)e^i_j$ is the orthonormal basis w.r.t $\hp$. Let $h(\tilde{e}^i_j)=h_{jk}\tilde{e}^i_k$, then by definition of compatibility, we have $h_{jk}\in L^\infty$, $\db h_{jk}\in L^2$. We follow the argument of Proof of \Cref{prop-subsheaf-andeg=pardeg}, and most of arguments are same. Assume 
    \begin{align*}
        f_k=c_{kj}\cdot \pi^*e^i_j\otimes f=\left(c_{kj}\prod_t||\sigma_t||^{\alpha^t_j}\right)\pi^*\tilde{e}^i_j\otimes f,
    \end{align*}
     then we have
     \begin{align*}
        \langle f_k,f_l\rangle_{H}=\sum_{j=1}^r\prod_{i=1}^m||\sigma_i||^{\alpha^i_p+\alpha^i_q}c_{kp}\bar{c}_{lq}h_{pq}.
    \end{align*}
    Since $h$ is a Hermitian matrix, there exists nonsingular matrix $P$ such that $h=PP^*$. Then
    \begin{align*}
        \prod_{i,j}||\sigma_i||^{-2\beta^i_j}\det K^H_0&=\sum_{1\leqs \tilde{j}_1<\cdots<\tilde{j}_s\leqs r}\left|\det\left(c_{k\tilde{j}_l}\prod_i||\sigma_i||^{\alpha^i_{\tilde{j}_l}-\alpha^i_{j_k}}\right)\right|^2\cdot\left|\det(p_{k\tilde{j}_l})\right|^2\\
        &=A+\sum_l B_l\prod_i||\sigma_i||^{2\gamma_{il}}.
    \end{align*}
    We still have the needed integrability conditions by $h_{ij}\in L^\infty$ and $\db h_{ij}\in L^2$.
\end{proof}
\begin{proposition}\label{prop-stablityequiv}
    $\eb$ is parabolic stable if and only if $(\eb,H)$ is analytic stable with respect to $\omega$ for any Hermitian metric $H$ compatible with the parabolic structure.
\end{proposition}
\begin{proof}
    If $\eb$ is parabolic stable, then by \Cref{prop-e-an=par} and \Cref{prop-subsheaf-an=par}, $(\eb,H)$ is analytic stable with $\omega$ for any $H$ compatible with parabolic structure. If $(\eb,H)$ is analytic stable, since the analytic stability only tests the coherent subsheaves of $\eb|_M$ rather than $\eb$, we need to prove an extension lemma as \cite[Section 6]{MR1701135}. Li-Narasimhan first use a theorem of Siu-Trautmann\cite[Theorem 2.2]{MR287033} to reduce the extension to a local problem, then proved the extension on a polydisc. Since the argument works for general complex manifolds, we can prove this in the same way.
\end{proof}

\section{Proof of Theorem \ref{thm-KHset}}

Now we are ready to prove the \Cref{thm-KHset}.
\begin{proof}[Proof of \Cref{thm-KHset}]
    Given $\eb$ is parabolic stable. By \Cref{prop-stablityequiv}, we have $(\eb,H_{par})$ is analytic stable with respect to $\omega$. By \Cref{thm-Zhangxi}, it suffices to check the conditions in the theorem. We know $|\Lambda_\omega F_{H_{par}}|_{H_{par}}$ is bounded on $M$ by \Cref{prop-tracecurvaturebdd}. Since $M$ is a Zariski open subset of the compact manifold $\mbar$, we can prove that assumptions (A1)-(A3) holds for the restriction metric $\omega$ by arguments similar to \cite[Proposition 2.2]{MR944577}. (A1) is trivial since $\vol(M,\omega|_M)=\vol(\mbar,\omega)<\infty$. For (A2), take $\phi=-\sum_i\log||\sigma_i||^2+C\geqs0$. $\phi$ is an exhaustion function and $\ii\dd\db\phi=-2\pi c_1(\oo_X(D))\in L^\infty(M)$. For (A3), each bounded function $f$ on $M$ with $\Delta f\geqs-B$ can be considered as a function on $\mbar$ that satisfies $\Delta f\geqs-B$ in distribution sense. Then the estimate follows from the standard Moser iteration.
    
    Since $|d\omega|_g\in C^\infty(\mbar)\subseteq L^2(M)$, we can apply \Cref{thm-Zhangxi} to get a Hermitian-Einstein metric $H$ on $\eb|_M$ compatible with the parabolic structure.

    Assume $H$ be a Hermitian-Einstein metric on $\eb|_M$ and compatible with the parabolic structure. Chern-Weil formula\cite[Lemma 3.2]{MR944577} tells us for any coherent subsheaf $\vb$ of $\eb|_M$,
    \begin{align*}
        d(\vb,H)=\ii\int_M\tr(\pi_V\Lambda_\omega F_H)-\int_M|\db\pi_V|^2_H,
    \end{align*}
    where $\pi_V$ is the orthogonal projection to $\vb$ with respect to $H$.
    Since $H$ is Hermitian-Einstein, $d(\vb,H)\leqs \frac{\rank \vb}{\rank \eb}d(\eb,H)$, i.e. analytic semistable. Moreover, the equality holds iff $\db\pi_V=0$ outside finite points. Assume there exists $\vb$ such that the equality holds. Since $\pi_V$ is holomorphic outside these points and bounded, it can be extended to a holomorphic section of $\End(\eb|_M)$.

    For any $p\in D$, there exists a neighborhood $U$ of $p$ and orthonormal basis $\{e_1,\cdots,e_r\}$ of $(\eb|_U,H_1)$ such that $\tilde{e}_j=\prod_i ||\sigma_i||^{-\alpha^i_j}e_j$ is the orthonormal basis of $H_{par}$. Thus there exists $a_j^k\in C^\infty(\bar{U})$ such that $\pi_V(\tilde{e}_j)=a_j^k\tilde{e}_k$. Then we have
    \begin{align*}
        \pi_V(e_j)=\prod_i ||\sigma_i||^{\alpha^i_j-\alpha^i_k}a_j^k e_k=:\tilde{a}_j^k e_k.
    \end{align*}
    Since $\alpha^i_j-\alpha^i_k>-1$, $\pi_V$ can be extended holomorphically from $U\backslash D$ to $U$.
    
    Since $\pi_V$ can be extended to a holomorphic projection endomorphism of $\eb$, this contradicts with the indecomposable condition.
    
\end{proof}


\section{Appendix}
\renewcommand{\thetheorem}{A.\arabic{theorem}}
\setcounter{section}{0} 

\begin{lemma}[Orthogonal foliation of tubular neighborhood]\label{lmm-foliation}
    Let $X$ be a compact complex manifold, $D\subset X$ is a smooth divisor. Then there exists a smooth foliation of some tubular neighborhood $N$ of $D$ such that each leaf is a complex disc. Moreover, if $\omega$ is a Hermitian metric on $X$, each leaf can be chosen to be orthogonal to $D$.
\end{lemma}
\begin{proof}
    For any $x\in D$, there exists a local coordinate $(U,z_1,z_*)$ near $x$ such that $W:=U\cap D=\{z_1=0\}$, where $z_*=(z_2,\cdots,z_n)$. Since $D$ is compact, we can assume $D=\bigcup_{i=1}^NV_i$, where $V_i\subset\subset W_i$. We will construct the foliation over $V=\bigcup_{i=1}^N V_i$ by induction.
    
    First, for $V_1$, the foliation can be defined by $L^{(1)}_p:=\{x\in U_1:z_*(x)=g_{p,*}(z_1)\}$ for any $p\in V_1$, where we can take $g_{p,*}(z_1)\equiv z_*(p)$. If the foliation $\fb_m$ is already constructed over $\bigcup_{i=1}^m V_i$, we will construct a new foliation $\fb_{m+1}$ over $\bigcup_{i=1}^{m+1}V_i$ which coincides $\fb_m$ over $\left(\bigcup_{i=1}^m V_i\right)\backslash W_{m+1}$. In the coordinate $(U_{m+1},z_1,z_*)$, the leaves of $\fb_m$ can be written as $L^{(m)}_p=\{(z_1,z_*):z_k=f_{p,k}(z_1),2\leqs k \leqs n\}$ for any $p\in \left(\bigcup_{i=1}^m V_i\right)\bigcap W_{m+1}$, where the functions $f_{p,k}$ are holomorphic w.r.t $z_1$ and smooth w.r.t $p$. Take a smooth function $\phi$ on $W_{m+1}$ such that $\phi|_{V_{m+1}}=1$, $\supp\phi\subset\subset W_{m+1}$. Then we define the new leaves to be $L_p^{(m+1)}:=\{(z_1,z_*):z_k=\tilde{f}_{p,k}(z_1),2\leqs k \leqs n\}$ for any $p\in \left(\bigcup_{i=1}^{m+1} V_i\right)\bigcap W_{m+1}$, where 
    \begin{align*}
        \tilde{f}_{p,k}(z_1):=(1-\phi(p))f_{p,k}(z_1)+\phi(p)g_{p,k}(z_1).
    \end{align*}
    This gives the new foliation $\fb_{m+1}$ over $\bigcup_{i=1}^{m+1} V_i$.

    Given a Hermitian metric $\omega$, it suffices to modify the leaf function from $g_{p,*}(z_1)\equiv z_*(p)$ to $g_{p,*}(z_1)=z_*(p)+z_1v_{p,*}$, where $(1,v_{p,2},\cdots,v_{p,n})\in T^{1,0}_pX$ and orthogonal to $T^{1,0}_p D$.
\end{proof}

\begin{proposition}[Poincar\'{e}-Lelong formula]\label{prop-poincare-lelong}
    Let $(X^{2n},J)$ be a compact complex manifold, $D\subset X$ be a divisor. Suppose $\sigma\in H^0(X,\oo_X(D))$ is a nontrivial global section which vanishes on $D$, $h$ is a Hermitian metric on $\oo_X(D)$. Then we have
    \begin{align*}
        \frac{\ii}{2\pi}\db\dd\log||\sigma||^2_h=c_1(\oo_X(D))-\delta_D
    \end{align*}
    in the current sense. As a consequence, we have $\dd\db\log||\sigma||_h^2\in L^\infty(X\backslash D)$ and
    \begin{align*}
        \int_{X\backslash D}\frac{\ii}{2\pi}\db\dd\log||\sigma||_h^2\wedge\omega^{n-1}=\int_X c_1^{BC}(\oo_X(D))\wedge\omega^{n-1}=\deg_\omega [D],
    \end{align*}
    for any Gauduchon metric $\omega$. 
\end{proposition}
\begin{proof}
    It suffices to show the case that $D$ is irreducible and reduced. Let $f:\oo_X(D)|_U\simeq U\times \cpxn$ be a trivialization of $\oo_X(D)$, then $f(\sigma|_U)=:\hat{\sigma}\in\oo(U)$. Let $\hat{h}:=f_*h\in C^\infty(U,\rean_+)$ be the Hermitian metric restricted to the trivialization such that $||\sigma||^2_h=|\hat{\sigma}|^2\cdot \hat{h}\in C^\infty(U)$. Then in $U$, we have 
    \begin{align*}
        \frac{\ii}{2\pi}\db\dd\log||\sigma||^2_h&=\frac{\ii}{2\pi}\db\dd\log|\hat{\sigma}|^2+\frac{\ii}{2\pi}\db\dd\log\hat{h}\\
        &=\frac{\ii}{2\pi}\db\dd\log|\hat{\sigma}|^2+\frac{\ii}{2\pi}\tr F_h.
    \end{align*}
     Thus it suffices to show that $\frac{\ii}{2\pi}\db\dd\log|\hat{\sigma}|^2=-\delta_D$, which is local problem and is a standard result, see \cite[(2.15)]{demailly1997complex}. 

     In $X\backslash D$, we always have $\dd\db\log|\hat{\sigma}|^2=\dd\db(\log\hat{\sigma}+\log\overline{\hat{\sigma}})=0$. Then $\db\dd\log||\sigma||^2=\tr F_h\in L^\infty(X\backslash D)$. Thus we have
     \begin{align*}
        \int_{X\backslash D}\frac{\ii}{2\pi}\db\dd\log||\sigma||_h^2\wedge\omega^{n-1}&=\int_{X\backslash D}\frac{\ii}{2\pi}\tr F_h\wedge\omega^{n-1}\\
        &=\int_{X}\frac{\ii}{2\pi}\tr F_h\wedge\omega^{n-1}\\
        &=\int_{X}c_1^{BC}(\oo_X(D))\wedge\omega^{n-1}.
     \end{align*}
\end{proof}

\begin{proposition}[\cite{MR909698}, Section 5.5]\label{prop-sheafregularity}
    Let $\fb$ be a coherent sheaf on a complex manifold $X$, then $\fb$ is locally free outside $Z\subset X$, where $\codim Z\geqs 1$. Moreover,
    \begin{enumerate}
        \item If $\fb$ is torsion free, then $\codim Z\geqs2$.
        \item If $\fb$ is reflexive, then $\codim Z\geqs 3$.
        \item If $\fb$ is reflexive and rank $1$, then $Z=\varnothing$, i.e. $\fb$ is a holomorphic line bundle. 
        \item If $\fb$ is subsheaf of a holomorphic vector bundle $\eb$ and the quotient $\eb/\fb$ is torsion free, then $\fb$ is reflexive.
    \end{enumerate}
\end{proposition}

\begin{proposition}\label{prop-projection}
    Suppose $(X,\omega)$ is a compact complex surface with a pluriclosed metric. Let $\tilde{X}:=\Bl_{x_0}X\xrightarrow{\pi} X$ be the blow-up of $X$, and $E:=\pi\iv(x_0)$ be the exceptional divisor. Then we have
    \begin{align*}
        \int_{\tilde{X}}c_1^{BC}(\oo_{\tilde{X}}(E))\wedge\pi^*\omega=0.
    \end{align*}
\end{proposition}
\begin{proof}
    If $(X,\omega)$ is K\"{a}hlerian, then it follows from the projection formula\cite[Lemma 1.1]{MR749574}. If $(X,\omega)$ is non-K\"{a}hlerian, we follow the argument in \cite[p.630]{MR939923}, \cite[p.182-187]{MR507725}.

    Since we have $\dd\db\omega=0$, then there exists a neighborhood $U$ of $x_0$, such that $\omega=\dd\alpha+\db\beta$ in $U$ for some $\alpha\in\Lambda^{0,1}$, $\beta\in\Lambda^{1,0}$. Let $\tilde{U}:=\pi\iv(U)$. By the construction of the blow-up, there exists a projection $p:\tilde{U}\to E\simeq \proj^1$, after possibly shrinking $U$. Let $\sigma\in H^0(\tilde{X},\oo_{\tilde{X}}(E))$ be the canonical section, then $E$ is exactly the vanishing set of $\sigma$. Let $h_0$ be the Hermitian metric on $\oo_{\tilde{X}}(E)$ over $\tilde{X}\backslash E$ such that $|s|_{h_0}\equiv 1$. Take $h_1$ be the Hermitian metric on $\oo_{\tilde{X}}(E)$ over $\tilde{U}$ by pulling backward the standard metric over $\proj^1$. Then we take $h:=(1-\psi)h_0+\psi h_1$, where $\psi$ is a cut-off function with $\supp \psi\subset\tilde{U}$ and $\psi=1$ near $E$.
    \begin{align*}
        \int_{\tilde{X}}c_1^{BC}(\oo_{\tilde{X}}(E))\wedge\pi^*\omega
        &=\frac{\ii}{2\pi}\int_{\tilde{X}}\db\dd\log h\wedge\pi^*\omega\\
        &=\frac{\ii}{2\pi}\int_{\tilde{U}}\db\dd\log h\wedge\pi^*\omega\\
        &=\frac{\ii}{2\pi}\int_{\tilde{U}}\db\dd\log h\wedge\pi^*(\dd\alpha+\db\beta)\\
        &=\frac{\ii}{2\pi}\int_{\tilde{X}}\db\dd\log h\wedge(\dd\pi^*\alpha+\db\pi^*\beta)\\
        &=0.
    \end{align*}
\end{proof}

\bibliographystyle{alpha}
\bibliography{ref.bib}
\end{document}